\documentclass[12pt]{article}

\usepackage[utf8]{inputenc}
\usepackage{geometry}
\usepackage{indentfirst}
\usepackage{graphicx}
\usepackage{subcaption}
\usepackage{amsthm}
\usepackage{amssymb}
\usepackage{amsmath}
\usepackage{blindtext} 
\usepackage{xcolor}
\usepackage{hyperref}
\hypersetup{
	citebordercolor=green,
	linkbordercolor=red}
\usepackage{amsmath,amssymb}
\usepackage{tikz}
\usepackage{color}
\usepackage[toc]{appendix}
\usepackage{graphicx}
\usepackage{fancyhdr}
\usepackage{enumitem}
\usepackage{bbm}
\usepackage{graphicx}
\usepackage{parskip}
\usepackage{float}
\usepackage{chngpage}
\usepackage{calc}
\usepackage{bigints}
\usepackage{tikz}
\usepackage{array}
\usepackage{bbm}
\usepackage{booktabs}
\usepackage{rotating}
\usepackage{multirow}
\usepackage{adjustbox}
\usepackage{tabularx}
\usepackage{enumitem}
\usepackage{verbatim}
\usepackage{mathtools}
\usepackage{ragged2e}
\usepackage[makeroom]{cancel}
\usepackage{enumitem}
\usepackage{caption}
\usepackage{mathtools}
\usepackage{hyperref}
\usepackage{bbm}
\usepackage{amsmath,amsfonts,amsthm,bm}
\usepackage{amsmath,amssymb}
\usepackage{appendix}
\usepackage{graphicx}
\usepackage{pgfplots}
\usepackage{tikz}
\usepackage{verbatim}
\usepackage{amsthm}
\usepackage[english]{babel}
\usepackage{hyphenat}
\usepackage[makeindex]{imakeidx}
\usepackage{tikz}

\usetikzlibrary{datavisualization}
\usetikzlibrary{matrix}
\usetikzlibrary{datavisualization.formats.functions}

\numberwithin{equation}{section}

\newtheorem{theorem}{Theorem}[section]
\newtheorem{proposition}[theorem]{Proposition}

\newtheorem{remark}[theorem]{Remark}

\makeatletter
\def\section{\@startsection {section}{1}{\z@}{3.25ex plus 1ex minus
		.2ex}{1.5ex plus .2ex}{\large\bf}}
\def\subsection{\@startsection{subsection}{2}{\z@}{3.25ex plus 1ex minus
		.2ex}{1.5ex plus .2ex}{\normalsize\bf}}
\@addtoreset{equation}{section} 
\makeatother

\title{On a new method for the stochastic perturbation of the disease transmission coefficient in SIS Models}

\author{Alberto Lanconelli\thanks{Dipartimento di Scienze Statistiche Paolo Fortunati, Università di Bologna, Bologna, Italy. \textbf{e-mail}: alberto.lanconelli2@unibo.it} \and  Berk Tan Perçin\thanks{Dipartimento di Fisica e Astronomia Augusto Righi, Università di Bologna, Bologna, Italy. \textbf{e-mail}: berktan.percin@studio.unibo.it }}

\date{\today}

\begin{document}
	
	\maketitle
	
	\bigskip
	
	\begin{abstract}
		In this study we investigate a novel approach to stochastically perturb the disease transmission coefficient, which is a key parameter in susceptible-infected-susceptible (SIS) models. Motivated by the papers \cite{Gray et. al.} and \cite{BL}, we perturb the disease transmission coefficient with a Gaussian white noise, formally modelled as the time derivative of a mean reverting Ornstein-Uhlenbeck process. We remark that, thanks to a suitable representation of the solution to the deterministic SIS model, this perturbation is rigorous and supported by a Wong-Zakai approximation argument that consists in smoothing the singular Gaussian white noise and then taking limit of the solution from the approximated model. We prove that the stochastic version of the classic SIS model obtained this way preserves a crucial feature of the deterministic equation: the reproduction number dictating the two possible asymptotic regimes for the infection, i.e. extinction and persistence, remains unchanged. We then identify the class of perturbing noises for which this property holds and propose simple sufficient conditions for that.
		All the theoretical discoveries are illustrated and discussed with the help of several numerical simulations. 
	\end{abstract}
	
Key words and phrases: SIS epidemic model, It\^o and Stratonovich stochastic differential equations, Wong-Zakai approximation, extinction, persistence. \\

AMS 2000 classification: 60H10, 60H30, 92D30.

\allowdisplaybreaks
	
\section{Introduction}

There are key biological parameters to analyse in epidemiology. The pathogen's type, infection target, pathogen's vulnerability to immune system, contagiousness and its life time outside of the host are certainly some of these key parameters. However, they are not enough to fully describe how an epidemic evolves. A key element, which makes epidemiology a very complex area, is the interaction inside the population and its unpredictability. 
\par To address this feature several mathematical models have been proposed and analysed in the literature: see e.g. \cite{Gnorrhea, Gray et. al., Vaccination Paper, SIRS Paper} and the references quoted there. In the classic susceptible-infected-susceptible (SIS) model there is no long term immunity for the infection. One example of such case is Gonorrhea, as described in \cite{Gnorrhea}, where individuals who are recovered from the infection become infected again and again. This means that long term immunity is not effective in prevention from re-infection. 
\par From a mathematical point of view, the SIS model (see e.g. \cite{Brauer}) is a very handy tool that describes the average evolution of an infection with no immunity. It consists of the system of ordinary differential equations
\begin{align}\label{Deterministic SIS Model}
\begin{cases}
\frac{dS(t)}{dt}=\mu N-\beta S(t)I(t)+\gamma I(t)-\mu S(t),& S(0)=s_0\in]0,N[;\\
\frac{dI(t)}{dt}=\beta S(t)I(t)-(\mu+\gamma) I(t),& I(0)=i_0\in]0,N[,
\end{cases}
\end{align}
where $S(t)$ and $I(t)$ denote the number of \emph{susceptibles} and \emph{infecteds} at time $t$, respectively. Here, $N:=s_0+i_0$ is the initial size of the population amongst whom the disease is spreading, $\mu$ denotes the per capita death rate, $\gamma$ is the rate at which infected individuals become cured and $\beta$ stands for the disease transmission coefficient. If we sum the equations in (\ref{Deterministic SIS Model}), we get that 
\begin{align*}
\frac{d}{dt}(S(t)+I(t))=\mu(N-(S(t)+I(t))), \quad S(0)+I(0)=N,
\end{align*}
which yields
\begin{align*}
S(t)+I(t)=S(0)+I(0)=N,\quad\mbox{ for all $t\geq 0$}.
\end{align*}
Therefore, system (\ref{Deterministic SIS Model}) can be reduced to the differential equation
\begin{align}\label{SIS one}
\frac{dI(t)}{dt}=\beta I(t)(N-I(t))-(\mu+\gamma) I(t),\quad I(0)=i_0\in ]0,N[,
\end{align}
with $S(t):=N-I(t)$. Furthermore, equation (\ref{SIS one}) can be solved explicitly as
\begin{align}\label{Deterministic Solution I(t)}
I(t)=\frac{i_0e^{[\beta N-(\mu+\gamma)]t}}{1+\beta\int_0^ti_0e^{[\beta N-(\mu+\gamma)]s}ds},\quad t\geq 0.
\end{align}
This explicit representation easily identifies two different asymptotic regimes for the solution $I(t)$, namely
\begin{align}\label{asymptotic determinstic}
\lim_{t\to+\infty}I(t)=
\begin{cases}
0,&\mbox{ if $R^D_0\leq 1$};\\
N\left(1-\frac{1}{R_0^D}\right),&\mbox{ if $R^D_0>1$},
\end{cases}
\end{align}
where 
\begin{align*}
R^D_0:=\frac{\beta N}{\mu+\gamma}.
\end{align*} 
This ratio is known as \emph{basic reproduction number} of the infection and determines whether the disease will become extinct, i.e. $I(t)$ will tend to zero as $t$ goes to infinity, or will be persistent, i.e. $I(t)$ will tend to a positive limit as $t$ increases.

\par As seen from \eqref{Deterministic Solution I(t)} and identity $S(t)=N-I(t)$, the infected and susceptible populations will have \emph{smooth} flows to each other with rates specified by the model. This description is very good at yielding overall results for a population but it fails to capture its heterogeneity, hence realism. In reality, individuals have different recovering rates or get infected at different rates, thus creating a distortion in the smooth flow between the two populations. To describe this randomness, several approaches have been proposed in the literature.\\ One example is the approach of Allen \cite{Allen}: here one starts with a discrete Markov chain whose transition probabilities reflect the dynamical behaviour of the deterministic model; then, via a suitable scaling on the one-step transition probability, one obtains a forward Fokker-Planck equation which is canonically associated with a stochastic differential equation.\\
Another common method for introducing stochasticity is the so-called \emph{parameter perturbation} approach \cite{Gray et. al., Feller's Test, Chinese Mean Reverting Paper}: it amounts at perturbing one of the parameters of the model equation with a suitable source of randomness, usually a Gaussian white noise. One of the most representative papers in this direction is \cite{Gray et. al.}: here the authors formally perturb equation (\ref{SIS one}), rewritten in the form
\begin{align*}
dI(t)=\beta I(t)(N-I(t))dt-(\mu+\gamma) I(t)dt,\quad I(0)=i_0\in ]0,N[,
\end{align*}
through the replacement 
\begin{align}
\label{Gray Substitution}
\beta dt\mapsto\beta dt + \sigma dB_t
\end{align}
with $\{B_t\}_{t\geq 0}$ being a standard one dimensional Brownian motion and $\sigma$ an additional parameter of the model. Since $\beta$ is a parameter for disease transmission rate, the term $\beta dt$ can be interpreted as number of transmissions in time interval $[t,t+dt]$, as stated in \cite{Gray et. al.}. This way, the authors propose the model
\begin{align}\label{Mao SDE}
dI(t)=[\beta I(t)(N-I(t))-(\mu+\gamma) I(t)]dt+\sigma I(t)(N-I(t))dB(t),
\end{align}
interpreted as an It\^o-type stochastic differential equation, which will encapsulate the randomness in the disease transmission; moreover, they identify a \emph{stochastic reproduction number} 
\begin{align*}
R_0^S:=R^D_0-\frac{\sigma^2N^2}{2(\mu+\gamma)},
\end{align*}
which, in contrast to (\ref{asymptotic determinstic}), characterizes the following asymptotic behaviours: 
\begin{itemize}
	\item if $R_0^S<1$ and $\sigma^2<\frac{\beta}{N}$ or if $\sigma^2>\max\{\frac{\beta}{N},\frac{\beta^2}{2(\mu+\gamma)}\}$, then the \emph{disease will become extinct}, i.e.
	\begin{align*}
	\lim_{t\to+\infty}I(t)=0;
	\end{align*}
	\item if $R_0^S>1$, then \emph{the disease will be persistent}, i.e.
	\begin{align*}
	\liminf_{t\to+\infty}I(t)\leq \xi\leq \limsup_{t\to+\infty}I(t),
	\end{align*}
	where $\xi:=\frac{1}{\sigma^2}\left(\sqrt{\beta^2-2\sigma^2(\mu+\gamma)}-(\beta-\sigma^2N)\right)$.
\end{itemize}
(see also \cite{Feller's Test}). It is worth mentioning that going from (\ref{SIS one}) to (\ref{Mao SDE}), as described in \cite{Gray et. al.}, one has to accept some reasonable but heuristic manipulations of the infinitesimal quantities $dt$ and $dB_t$.  

\par The aim of this paper is to propose a different method for perturbing the disease transmission rate in the SIS model (\ref{SIS one}). Our idea stems from the following simple observation: if we let $\beta$ in (\ref{SIS one}) to be a function of time, then the solution formula (\ref{Deterministic Solution I(t)}) takes the form
\begin{align}\label{I(t) Solution Only Integral of Beta}
I(t)&=\frac{i_0e^{N\int_0^t\beta(s) ds-(\mu+\gamma)t}}{1+i_0\int_0^t\beta(s)e^{N\int_0^s\beta(r)dr -(\mu+\gamma)s}ds}\nonumber\\
&=\frac{i_0e^{N\int_0^t\beta(s) ds-(\mu+\gamma)t}}{1+\frac{i_0}{N}\left(e^{N\int_0^t\beta(r)dr -(\mu+\gamma)t}-1+\int_0^te^{N\int\limits_0^s\beta(r)dr-(\gamma+\mu)s}(\gamma+\mu)ds\right)},
\end{align}
where in the second equality we performed an integration by parts in the denominator.
Equation \eqref{I(t) Solution Only Integral of Beta} now depends on the function $\beta(t)$ only through its integral $\int_0^t\beta(s)ds$. In \cite{BL} the authors utilized this approach to mimic the perturbation proposed in \cite{Gray et. al.}; in this case, the \emph{singular} perturbation 
\begin{align*}
\beta(t)\mapsto\beta(t)+\sigma\frac{dB_t}{dt},
\end{align*}
formally employed on the differential equation (\ref{SIS one}) by the authors in \cite{Gray et. al.}, becomes the well defined transformation 
\begin{align*}
\int_0^t\beta(s)ds\mapsto\int_0^t\beta(s)ds+\sigma B_t,
\end{align*} 
if directly applied on the explicit solution (\ref{I(t) Solution Only Integral of Beta}). As shown in \cite{BL}, this different procedure of parameter perturbation results in an alternative stochastic SIS model which surprisingly exhibits the same asymptotic regimes of its deterministic counterpart (\ref{SIS one}). It is important to remark that this new parameter perturbation approach, which directly acts on the explicit solution (\ref{I(t) Solution Only Integral of Beta}), is also justified via Wong-Zakai theorem at the level of differential equations, thus ruling out the necessity of having a closed form expression for the solution. \\
In the current paper we employ the just mentioned approach to the case where the perturbation is modelled as a mean reverting Ornstein -Uhlenbeck process. This choice is suggested, but not investigated, both in \cite{Allen} and \cite{Gray et. al.}. From a modelling point of view it is motivated by the fact that the variance of a mean reverting Ornstein -Uhlenbeck process is bounded in time, while the one of a Brownian motion, utilized in \cite{Gray et. al.}, is not. This feature seems to be more realistic and hence desirable. However, from a mathematical point of view, the perturbation with a mean reverting Ornstein -Uhlenbeck process makes the analysis of the model more demanding since in this case equation  \eqref{Mao SDE} becomes a stochastic differential equation with random coefficients.\\    
This problem is discussed in \cite{Chinese Mean Reverting Paper} with an approach that follows \cite{Gray et. al.}. Here, we introduce the model working directly on the explicit representation (\ref{I(t) Solution Only Integral of Beta}) and cross-validate the proposal from a differential equations' perspective passing through the Wong-Zakai theorem. We prove that our model fulfils some basic biological constraints, i.e. the solution is global and lives in the interval $]0,N[$ with probability one. Then, we analyse the asymptotic behaviour and discover that the threshold for the different regimes coincides with the one for the deterministic SIS model; in other words, the parameters describing the mean reverting Ornstein -Uhlenbeck process do not play any role in the limiting behaviour of the solution. We also identify a class of perturbations for which this invariance is preserved thus offering a complete analysis of our approach.

The paper is organized as follows: in Section 2 we introduce the model and its cross-validation via the Wong-Zakai theorem; Section 3 is devoted to the analysis of our model: support of the solution, extinction, persistence and discussion of several numerical simulations. Lastly, in Section 4 we address the problem of finding a general class of perturbations for which the results from Section 3 remain the same; numerical simulations are also presented for this enlarged framework.

\section{Stochastic parameter perturbation with a mean reverting Ornstein-Uhlenbeck process}

Let $\{Y_t\}_{t\geq 0}$ be a mean reverting Ornstein-Uhlenbeck process driven by a standard one dimensional Brownian motion $\{B_t\}_{t\geq 0}$; this means that $\{Y_t\}_{t\geq 0}$ is the unique strong solution of the stochastic differential equation
\begin{align}\label{OU}
dY_t=-\alpha Y_tdt+\sigma dB_t,\quad Y_0=0
\end{align} 
where the parameters $\alpha$ and $\sigma$ are positive real number. the process $\{Y_t\}_{t\geq 0}$ can be explicitly represented as
\begin{align}\label{explicit OU}
Y_t=\sigma\int\limits_0^te^{-\alpha(t-s)}dB_s,\quad t\geq 0,
\end{align}
entailing that $Y_t$ is a Gaussian random variable with mean zero and variance $\frac{1-e^{-2\alpha t}}{2\alpha}\sigma^2$. We also recall the ergodic property of $\{Y_t\}_{t\geq 0}$:
\begin{align}\label{ergodic}
\lim_{t\to +\infty}\frac{1}{t}\int_0^tY_sds=0\quad\mbox{almost surely}.
\end{align}
We now perturb (\ref{I(t) Solution Only Integral of Beta}) via the substitution
\begin{align}\label{beta+Y singular}
\beta(t)\mapsto \beta+\frac{dY_t}{dt}
\end{align}
or more rigorously
\begin{align}\label{beta+Y}
\int_0^t\beta(s)ds\mapsto \int_0^t\left(\beta+\frac{dY_s}{ds}\right)ds=\beta t+Y_t.
\end{align}
This gives
\begin{align}\label{Perturbed I(t)}
	\mathtt{I}_t:=\frac{i_0e^{\nu t+NY_t}}{1+\frac{i_0}{N}\left(e^{\nu t+NY_t}-1+\int_0^te^{\nu s+NY_s}(\gamma+\mu)ds\right)},\quad t\geq 0,
\end{align}
where to ease the notation we set $\nu:=N\beta-(\gamma+\mu)$ and $\mathtt{I}_t$ instead of $I(t,Y_t)$; note that $R_0^D\leq 1$ is equivalent to $\nu\leq 0$. The stochastic process (\ref{Perturbed I(t)}) is the object of our investigation. Observe that an application of the It\^o formula gives
\begin{align}\label{dI(t)}
	d\mathtt{I}_t=&\left[\mathtt{I}_t(N-\mathtt{I}_t)\left(\frac{\nu}{N}-\alpha Y_t+\frac{N-2\mathtt{I}_t}{2}\sigma^2\right)-\frac{\gamma+\mu}{N}\mathtt{I}_t^2\right]dt+\sigma\mathtt{I}_t(N-\mathtt{I}_t)dB_t.
\end{align}	
This equation can be considered either as a one dimensional stochastic differential equation with random coefficients (for the presence of $\{Y_t\}_{t\geq 0}$) or, if coupled with (\ref{OU}), as a two dimensional system of stochastic differential equations. The local Lipschitz continuity of the coefficients of such system entails path-wise uniqueness and hence that the couple $\{(\mathtt{I}_t,Y_t)\}_{t\geq 0}$, with $\mathtt{I}_t$ defined in (\ref{Perturbed I(t)}) and $Y_t$ defined in (\ref{explicit OU}), is its unique solution (see e.g. Theorem 2.5, Chapter 5 in \cite{KS}).   

\subsection{Cross-validation of the model via Wong-Zakai theorem}

We obtained the stochastic process (\ref{Perturbed I(t)}) perturbing the explicit solution (\ref{I(t) Solution Only Integral of Beta}) with the transformation (\ref{beta+Y}). One can however derive the stochastic differential equation (\ref{dI(t)}), which is uniquely solved by (\ref{Perturbed I(t)}), through a parameter perturbation procedure acting on the deterministic equation (\ref{SIS one}), which resembles the approach employed in \cite{Gray et. al.}.\\
Let $\{B^{\pi}_t\}_{t\in [0,T]}$ be the polygonal approximation of the Brownian motion $\{B_t\}_{t\in [0,T]}$, relative to the partition $\pi$. This means that $\{B^{\pi}_t\}_{t\in [0,T]}$ is a continuous piecewise linear random function converging to $\{B_t\}_{t\in [0,T]}$ almost surely and uniformly on $[0,T]$, as the mesh of the partition tends to zero. Now, replace (\ref{OU}) with
\begin{align}\label{Wong-Zakai Perturbation}
	\frac{dY^{\pi}_t}{dt} = -\alpha Y^{\pi}_t+\sigma \frac{dB^{\pi}_t}{dt},\quad Y_0^{\pi}=0,
	\end{align}
which gives a smooth approximation of $\{Y_t\}_{t\geq 0}$. Using $\{Y^{\pi}_t\}_{t\geq 0}$ instead of $\{Y_t\}_{t\geq 0}$ allows for a rigorous implementation of the transformation (\ref{beta+Y singular}) in (\ref{SIS one}), that means
\begin{align}\label{SIS one with Y}
\frac{d\mathtt{I}_t^{\pi}}{dt}=&\beta \mathtt{I}_t^{\pi}(N-\mathtt{I}_t^{\pi})-(\mu+\gamma) \mathtt{I}_t^{\pi}+\mathtt{I}_t^{\pi}(N-\mathtt{I}_t^{\pi})\frac{dY^{\pi}_t}{dt}\nonumber\\
=&\beta \mathtt{I}_t^{\pi}(N-\mathtt{I}_t^{\pi})-(\mu+\gamma) \mathtt{I}_t^{\pi}-\alpha \mathtt{I}_t^{\pi}(N-\mathtt{I}_t^{\pi})Y^{\pi}_t\nonumber\\
&+\sigma \mathtt{I}_t^{\pi}(N-\mathtt{I}_t^{\pi})\frac{dB^{\pi}_t}{dt}.
\end{align}
According to the Wong-Zakai Theorem \cite{Wong-Zakai Paper},\cite{Stroock Varadhan} 
the unique solution $\{\mathtt{I}_t^{\pi}\}_{t\in [0,T]}$ of the random ordinary differential equation (\ref{SIS one with Y}) converges, as the mesh of the partition $\pi$ tends to zero, to the solution of the Stratonovich stochastic differential equation 
\begin{align*}
d\mathtt{I}_t=&\left[\beta \mathtt{I}_t(N-\mathtt{I}_t)-(\mu+\gamma) \mathtt{I}_t-\alpha \mathtt{I}_t(N-\mathtt{I}_t)Y_t\right]dt+\sigma \mathtt{I}_t(N-\mathtt{I}_t)\circ dB_t,
\end{align*}
which in turn is equivalent to the It\^o SDE
\begin{align*}
d\mathtt{I}_t=&\left[\beta \mathtt{I}_t(N-\mathtt{I}_t)-(\mu+\gamma) \mathtt{I}_t-\alpha \mathtt{I}_t(N-\mathtt{I}_t)Y_t+\frac{\sigma^2}{2}\mathtt{I}(t)(N-\mathtt{I}(t))(N-2\mathtt{I}(t))\right]dt\\
&+\sigma \mathtt{I}_t(N-\mathtt{I}_t)dB_t.
\end{align*}
The stochastic differential equation above coincides with (\ref{dI(t)}) thus validating our parameter perturbation approach also from a model equation pont of view.

\section{Analysis of the stochastically perturbed SIS Model}

In this section we analyse the stochastic process \eqref{Perturbed I(t)} which we recall to be the unique strong solution of the SDE \eqref{dI(t)}. We will in particular show that such process lives in the interval $]0,N[$, for all $t\geq 0$, almost surely and we will provide sufficient conditions for extinction and persistence.

\subsection{Support of the solution}

We start with the following.

\begin{proposition}\label{I(t) in (0,N)}
	For the stochastic process $\{\mathtt{I}_t\}_{t\geq 0}$ defined in \eqref{Perturbed I(t)} we have 
	\begin{align*}
	\mathbb{P}(\mathtt{I}_t\in ]0,N[)=1,\quad\mbox{for all $t\geq 0$}.
	\end{align*}
\end{proposition}

\begin{proof}
	First of all, we observe that $\mathtt{I}_t$ can be rewritten as
	\begin{align*}
	\mathtt{I}_t=\frac{Ni_0e^{\nu t+NY_t}}{N-i_0+i_0e^{\nu t+NY_t}+i_0\int_0^te^{\nu s+NY_s}(\gamma+\mu)ds}.
	\end{align*}
	Since by assumption $0<i_0<N$, we see that $\mathtt{I}_t$ is a ratio of almost sure positive quantities; this yields $\mathtt{I}_t>0$ for all $t\geq 0 $ almost surely. On the other hand, the last identity also gives 
	\begin{align*}
	\mathtt{I}_t&<\frac{Ni_0e^{\nu t+NY_t}}{i_0e^{\nu t+NY_t}+i_0\int_0^te^{\nu s+NY_s}(\gamma+\mu)ds}\\
	&=\frac{N}{1+(\gamma+\mu)\int\limits_0^te^{\nu(s-t)+N(Y_s-Y_t)}ds}\\
	&< N.
	\end{align*}
	The proof is complete.
\end{proof}

\subsection{Extinction of the infection}

We now provide a sufficient condition for extinction; remarkably, the parameters describing the stochastic perturbation, i.e. $\alpha$ and $\sigma$, do not play role in that. 

\begin{theorem}\label{Extinction Theorem}
	If $R^D_0=\frac{\beta N}{\gamma+\mu}\leq 1$, or equivalently $\nu=\beta N-(\gamma+\nu)\leq 0$, then 
	\begin{align*}
	\lim_{t\to\infty}\mathtt{I}_t=0\quad\mbox{ almost surely}.
	\end{align*}
\end{theorem}

\begin{proof}
	We take $G(x):=\ln\left(\frac{x}{N-x}\right)$ for $x\in]0,N[$ and observe that $G$ is a strictly increasing function that maps the interval $]0,N[$ into $]-\infty,+\infty[$. An application of the It\^o formula gives 
	\begin{align*}
	dG(\mathtt{I}_t)=\left[\nu-N\alpha Y_t-(\gamma+\mu)\frac{\mathtt{I}_t}{N-\mathtt{I}_t}\right]dt+\sigma N dB_t,
	\end{align*}
	which corresponds to the integral equation
	\begin{align}\label{G(I)}
		G(\mathtt{I}_t)=\ln\left(\frac{i_0}{N-i_0}\right)-\alpha N\int\limits_0^tY_sds+\int\limits_0^tf(\mathtt{I}_s)ds+\sigma NB_t,
	\end{align}
	where
	\begin{align}\label{f(x)}
		f(x):=\nu-(\gamma+\mu)\frac{x}{N-x},\quad \mbox{ for $x\in]0,N[$}.
	\end{align}
	It is useful to note that $f$ is monotone decreasing on the interval $]0,N[$ and that $f(x)<\nu$, for $x\in]0,N[$. Therefore, from equation \eqref{G(I)} we get
	\begin{align}\label{II}
		\limsup_{t\to\infty}\frac{1}{t}G(\mathtt{I}_t)\leq&\limsup_{t\to\infty}\frac{1}{t}\ln\left(\frac{i_0}{N-i_0}\right)-\lim_{t\to\infty}\frac{\alpha N}{t}\int\limits_0^tY_sds\nonumber\\
		&+\limsup_{t\to\infty}\frac{1}{t}\int\limits_0^tf(\mathtt{I}_s)ds+\sigma N\limsup_{t\to\infty}\frac{B_t}{t}.
	\end{align}
	By recalling (\ref{ergodic}), it is easy to notice that the first two terms in right hand side above are equal to zero. Moreover, by the strong law of large numbers for martingales (see for instance \cite{Mao}) we also have
	\begin{align*}
	\lim_{t\to\infty}\frac{B_t}{t}=0,\quad\mbox{almost surely}.
	\end{align*}
	Therefore, inequality (\ref{II}) now reads
	\begin{align*}
	\limsup_{t\to\infty}\frac{1}{t}G(\mathtt{I}_t)\leq\limsup_{t\to\infty}\frac{1}{t}\int\limits_0^tf(\mathtt{I}_s)ds<\nu,
	\end{align*}
	that means 
	\begin{align*}
	\limsup_{t\to\infty}\frac{1}{t}\ln\left(\frac{\mathtt{I}_t}{N-\mathtt{I}_t}\right)<0,\quad\mbox{almost surely}.
	\end{align*}
	Since the last statement implies our thesis, the proof is complete.
\end{proof}

In the stochastic SIS model obtained by this parameter perturbation method from the deterministic one, the limiting behaviour of both models are the same for $R^D_0\leq1$.

\subsection{Persistence of the Infection}

We now turn to the problem of finding sufficient condition for persistence of the disease. Again, the parameters describing the stochastic perturbation, i.e. $\alpha$ and $\sigma$, do not influence the threshold. 

\begin{theorem}\label{Persistency Theorem}
	If $R^D_0>1$, or equivalently $\nu>0$, then we have with probability one
	\begin{align}\label{persistence}
	\limsup_{t\to\infty}\mathtt{I}_t\geq x^* \quad\mbox{ and } \quad \liminf_{t\to\infty}\mathtt{I}_t\leq x^*,
	\end{align}
	where $x^*=N\left(1-\frac{1}{R^D_0}\right)$.
\end{theorem}

\begin{proof}
	The assumption of the theorem implies that the function in \eqref{f(x)} has a unique root in $]0,N[$ given by $x^*=N(1-1/R^D_0)$, as shown in figure \ref{f(x) plot nu>0}.
	\begin{figure}[ht!]
		\centering
		\includegraphics[width=0.6\textwidth]{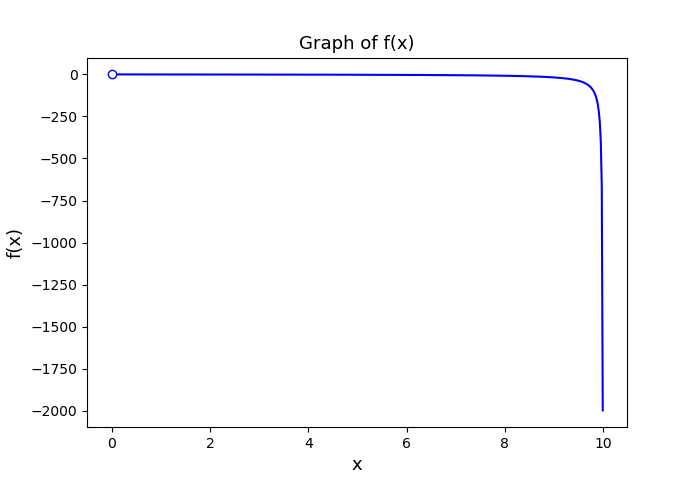}
		\caption{The plot of $f(x)$ from (\ref{f(x)}) with $\nu=20$, $\gamma+\mu=0.2$ and $N=10$}
		\label{f(x) plot nu>0}
	\end{figure}
	\par We follow the proof of Theorem 5.1 in \cite{Gray et. al.}. Assume the first inequality in (\ref{persistence}) to be false. Then, there exists $\varepsilon>0$ such that 
	\begin{align}
		\label{limsup contraty assertion}
		\mathbb{P}(\Omega_1)>\varepsilon\quad\mbox{ where }\quad \Omega_1:=\left\{\limsup_{t\to\infty}\mathtt{I}_t\leq x^*-\varepsilon\right\}.
	\end{align}
Therefore, for all $\omega\in\Omega_1$ there exists $T(\omega)\geq 0$ such that
\begin{align*} 
\mathtt{I}_t\leq x^*-\varepsilon,\quad\mbox{ for all $t\geq T(\omega)$},
\end{align*}
and the monotonicity of $f$ yields
	\begin{align}
		\label{limsup f(x) comparison}
		f(\mathtt{I}_t)\geq f(x^*-\varepsilon) >0,\quad\mbox{ for all $t\geq T(\omega)$}.
	\end{align}
Therefore, using identities (\ref{G(I)}), (\ref{ergodic}) and the strong law of large numbers for martingales, we can write for all $\omega\in\Omega_1$ that
	\begin{align*}
	\liminf_{t\to\infty}\frac{1}{t}G(\mathtt{I}_t)\geq&\liminf_{t\to\infty}\frac{1}{t}\ln\left(\frac{i_0}{N-i_0}\right)-\lim_{t\to\infty}\frac{\alpha N}{t}\int\limits_0^tY_sds\\
	&+\liminf_{t\to\infty}\frac{1}{t}\int\limits_0^tf(\mathtt{I}_s)ds+N\sigma\lim_{t\to\infty}\frac{B_t}{t}\\
	=&\liminf_{t\to\infty}\frac{1}{t}\int\limits_0^tf(\mathtt{I}_s)ds\\
	\geq&\liminf_{t\to\infty}\frac{1}{t}\int\limits_0^Tf(\mathtt{I}_s)ds+f(x^*-\varepsilon)\liminf_{t\to\infty}\frac{t-T}{t}
	\end{align*}
	and hence
	\begin{align*}
	\liminf_{t\to\infty}\frac{1}{t}\ln\left(\frac{\mathtt{I}_t}{N-\mathtt{I}_t}\right)\geq f(x^*-\varepsilon)>0.
	\end{align*}
	This gives $\lim\limits_{t\to\infty}\mathtt{I}_t=N$ which contradicts
	\eqref{limsup contraty assertion}.
	\par To prove the second inequality in (\ref{persistence}) we proceed as before and assume that there exists $\varepsilon>0$ such that
	\begin{align}
		\label{liminf contraty assertion}
		\mathbb{P}(\Omega_2)>\varepsilon\quad\mbox{ where }\quad \Omega_2=\left\{\liminf_{t\to\infty}\mathtt{I}_t\geq x^*+\varepsilon\right\}.
	\end{align}
	Therefore, for all $\omega\in\Omega_2$, there exists $T(\omega)\geq 0$ such that 
	\begin{align*}
	\mathtt{I}_t\geq x^*-\varepsilon,\quad\mbox{ for all $t\geq T(\omega)$}.
	\end{align*}
	The monotonicity of $f$ gives
	\begin{align}
		\label{liminf f(x) comparison}
		f(\mathtt{I}_t)\leq f(x^*-\varepsilon) <0,\quad \mbox{ for all $t\geq T(\omega)$},
	\end{align}
	and 
	\begin{align*}
	\limsup_{t\to\infty}\frac{1}{t}\ln\left(\frac{\mathtt{I}_t}{N-\mathtt{I}_t}\right)\leq\limsup_{t\to\infty}\frac{1}{t}\int\limits_0^tf(\mathtt{I}_s)ds\leq f(x^*+\varepsilon)<0.
	\end{align*}
	This implies $\lim\limits_{t\to\infty}\mathtt{I}_t=0$, contradicting \eqref{liminf contraty assertion}.
\end{proof}

\subsection{Trajectory simulations}
In this section we present various simulations; we consider two different values of $\sigma$ to emphasize that, according to our theoretical  results, the limiting behaviour of the solution doesn't depend on them.\\
For the first example, we set $N=200$, $i_0=100$, $\beta=0.06$, $\gamma+\mu=14$, $\alpha=0.4$, so that $R^D_0=0.8$ and $\nu=-2$. According to Theorem \ref{Extinction Theorem} the infection should extinct almost surely. See Figure \ref{R=0.8 nu=-2}.
\begin{figure}[h!]
	\centering
	\begin{subfigure}[c]{0.4\textwidth}
		\includegraphics[width=\linewidth]{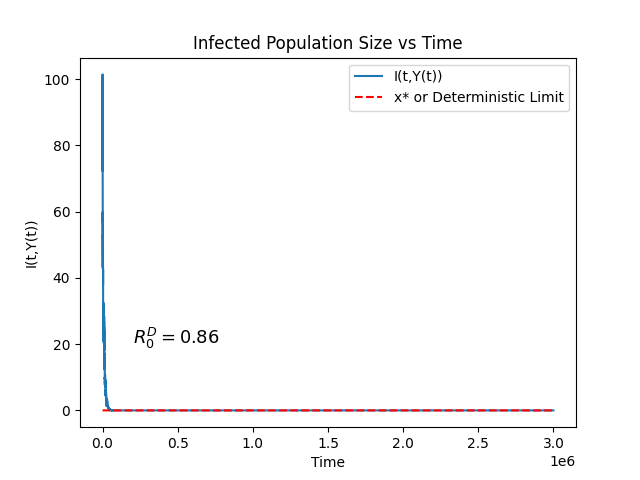}
		\caption{$\sigma=0.005$}
	\end{subfigure}
	\begin{subfigure}[c]{0.4\textwidth}
		\includegraphics[width=\linewidth]{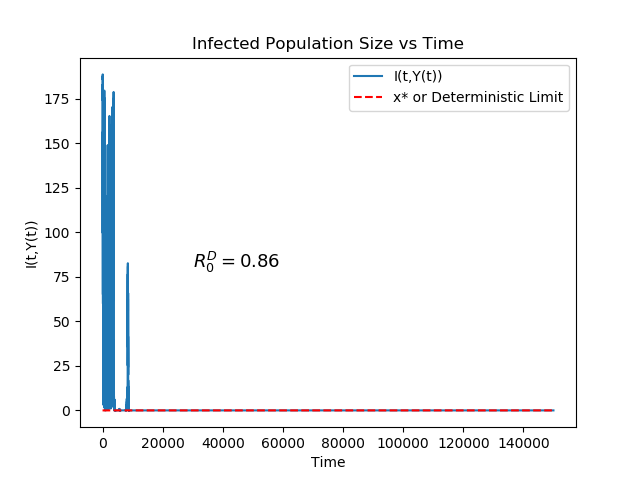}
		\caption{$\sigma=0.05$}
	\end{subfigure}
	\caption{The plot with parameters $N=200$, $i_0=100$, $\alpha=0.4$, $R^D_0=0.857$ and hence $\nu=-2$. The label for y-axis $I(t,Y_t)$ stands for $\mathtt{I}_t$.}
	\label{R=0.8 nu=-2}
\end{figure}
\begin{figure}[h!]
	\centering
	\begin{subfigure}[c]{0.4\textwidth}
		\includegraphics[width=\linewidth]{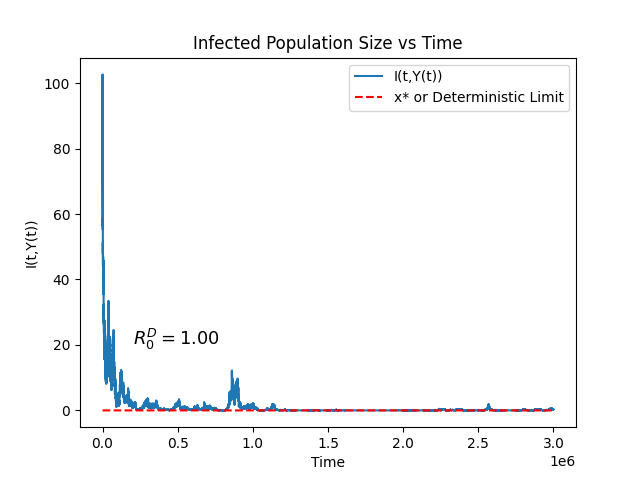}
		\caption{$\sigma=0.005$}
	\end{subfigure}
	\begin{subfigure}[c]{0.4\textwidth}
		\includegraphics[width=\linewidth]{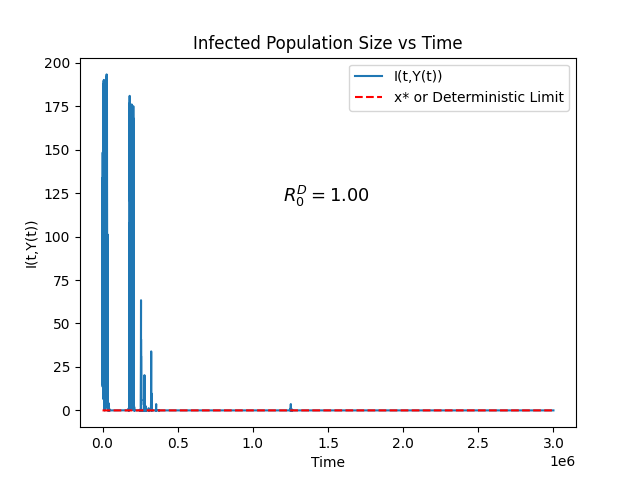}
		\caption{$\sigma=0.05$}
	\end{subfigure}
	\caption{The plot with parameters $N=200$, $i_0=100$, $\alpha=0.4$, $R^D_0=1$ and hence $\nu=0$. The label for y-axis $I(t,Y_t)$ stands for $\mathtt{I}_t$.}
	\label{R=1 nu=0}
\end{figure}
\begin{figure}[h!]
	\centering
	\begin{subfigure}[c]{0.4\textwidth}
		\includegraphics[width=\linewidth]{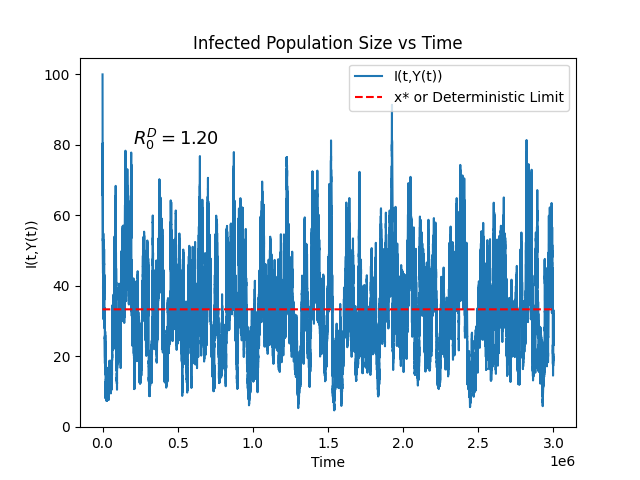}
		\caption{$\sigma=0.005$}
	\end{subfigure}
	\begin{subfigure}[c]{0.4\textwidth}
		\includegraphics[width=\linewidth]{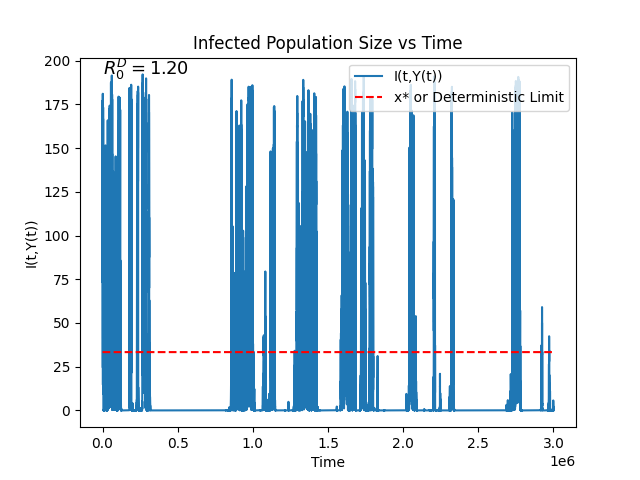}
		\caption{$\sigma=0.05$}
		\label{Gray non-valid condition}
	\end{subfigure}
	\caption{The plot with parameters $N=200$, $i_0=100$, $\alpha=0.4$ and $R^D_0=1.2$ so $\nu=2$. The deterministic limit is 33. The label for y-axis $I(t,Y_t)$ stands for $\mathtt{I}_t$.}
	\label{R=1.2 nu=2}
\end{figure}
\par Next, the same simulation is performed with a different set of parameters. We take $N=200$, $i_0=100$, $\beta=0.06$, $\gamma+\mu=12$, $\alpha=0.4$ which gives $R^D_0=1$ and $\nu=0$. According to Theorem \ref{Extinction Theorem} the infection should extinct almost surely. See Figure \ref{R=1 nu=0}.
\par Lastly, we consider $N=200$, $i_0=100$, $\beta=0.06$, $\gamma+\mu=10$, $\alpha=0.4$ so $R^D_0=1.2$ and $\nu=2$. According to Theorem \ref{Extinction Theorem} the infection should persist a.s., or more precisely it should oscillate above and below the deterministic limit. See Figure \ref{R=1.2 nu=2} where fluctuations above and below the deterministic limit $x^*$ is visible as well.

\section{Stochastic parameter perturbation with a general process}

In this section we try to understand to which extent the results of the previous sections are determined by the choice of the particular perturbation $\{Y_t\}_{t\geq 0}$ in (\ref{OU}). To this aim we consider $\{Z_t\}_{t\geq 0}$, solution of the stochastic differential equation
\begin{align}\label{Z}
dZ_t=b(t,Z_t)dt+\sigma dB_t,\quad Z_0=0,
\end{align}
where $\sigma$ is positive real number and the function $b:[0,T]\times\mathbb{R}$ is assumed to be globally Lipschitz continuous in $z$, uniformly in $t$; we also introduce the corresponding perturbation of the parameter $\beta$, namely 
\begin{align}\label{beta+Z}
\int_0^t\beta(s)ds\mapsto \int_0^t\left(\beta+\frac{dZ_s}{ds}\right)ds=\beta t+Z_t.
\end{align}
If we employ such transformation in (\ref{I(t) Solution Only Integral of Beta}), which is the solution of the deterministic SIS model with a time dependent transmission coefficient $\beta$, we get the stochastic process
\begin{align}\label{Perturbed Z(t)}
\mathcal{I}_t:=\frac{i_0e^{\nu t+NZ_t}}{1+\frac{i_0}{N}\left(e^{\nu t+NZ_t}-1+\int_0^te^{\nu s+NZ_s}(\gamma+\mu)ds\right)},\quad t\geq 0.
\end{align}
where to ease the notation we set $\mathcal{I}_t$ instead of $I(t,Z_t)$. Moreover, an application of the It\^o formula yields
\begin{align}\label{dI(Z)}
d\mathcal{I}_t=&\left[\mathcal{I}_t(N-\mathcal{I}_t)\left(\frac{\nu}{N}+b(t,Z_t)+\frac{\sigma^2N}{2}-\sigma^2\mathcal{I}_t\right)-\frac{\gamma+\mu}{N}\mathcal{I}_t^2\right]dt\nonumber\\
&+\sigma\mathcal{I}_t(N-\mathcal{I}_t)dB_t
\end{align}

\begin{remark}
The restriction to constant diffusion coefficients imposed in \eqref{Z} is due the necessity of cross validating the model (\ref{Perturbed Z(t)}) also from a differential equations' point of view. In fact, if we smooth the process $\{Z_t\}_{t\geq 0}$ as 
\begin{align*}
\frac{dZ^{\pi}_t}{dt}=b(t,Z^{\pi}_t)+\sigma \frac{dB^{\pi}_t}{dt},\quad Z^{\pi}_0=0,
\end{align*}
and perturb correspondingly the parameter $\beta$ in equation (\ref{SIS one}), we obtain
\begin{align*}
\frac{d\mathcal{I}_t^{\pi}}{dt}=&\beta \mathcal{I}_t^{\pi}(N-\mathcal{I}_t^{\pi})-(\mu+\gamma) \mathcal{I}_t^{\pi}+b(t,Z^{\pi}_t)\mathcal{I}_t^{\pi}(N-\mathcal{I}_t^{\pi})\nonumber\\
&+\sigma \mathcal{I}_t^{\pi}(N-\mathcal{I}_t^{\pi})\frac{dB^{\pi}_t}{dt}.
\end{align*}
According to the Wong-Zakai Theorem the unique solution $\{\mathcal{I}_t^{\pi}\}_{t\in [0,T]}$ of the random ordinary differential equation above converges, as the mesh of the partition $\pi$ tends to zero, to the solution of the Stratonovich stochastic differential equation 
\begin{align*}
d\mathcal{I}_t=&\left[\beta \mathcal{I}_t(N-\mathcal{I}_t)-(\mu+\gamma) \mathcal{I}_t+b(t,Z_t)\mathcal{I}_t(N-\mathcal{I}_t)\right]dt+\sigma \mathcal{I}_t(N-\mathcal{I}_t)\circ dB_t,
\end{align*}
which in turn is equivalent to the It\^o SDE
\begin{align*}
d\mathcal{I}_t=&\left[\beta \mathcal{I}_t(N-\mathcal{I}_t)-(\mu+\gamma) \mathcal{I}_t+b(t,Z_t) \mathcal{I}_t(N-\mathcal{I}_t)+\frac{\sigma^2}{2}\mathcal{I}(t)(N-\mathcal{I}(t))(N-2\mathcal{I}(t))\right]dt\\
&+\sigma \mathcal{I}_t(N-\mathcal{I}_t)dB_t.
\end{align*}
This SDE coincides with (\ref{dI(Z)}) thus validating that model. If we allow $\sigma$ in (\ref{Z}) to depend also on $Z$, then this match wouldn't take place for the presence of an additional drift term in the equation for $\{Z_t\}_{t\geq 0}$.
\end{remark}

We now start to analyse the properties of $\{\mathcal{I}_t\}_{t\geq 0}$ by stating the analogue of Proposition \ref{I(t) in (0,N)}.

\begin{proposition}
	For the stochastic process $\{\mathcal{I}_t\}_{t\geq 0}$ defined in \eqref{Perturbed Z(t)} we have 
	\begin{align*}
	\mathbb{P}(\mathcal{I}_t\in ]0,N[)=1,\quad\mbox{for all $t\geq 0$}.
	\end{align*}
\end{proposition}

\begin{proof}
	Looking through the proof of Proposition \ref{I(t) in (0,N)} one easily see that 
	the same conclusion holds for $\{\mathcal{I}_t\}_{t\geq 0}$.
\end{proof}

The next theorem provides a sufficient condition on the stochastic process $\{Z_t\}_{t\geq 0}$ which guarantees extinction for $\{\mathcal{I}_t\}_{t\geq 0}$.

\begin{theorem}\label{Generalization Extinction Theorem}
	Assume that $R^D_0=\frac{\beta N}{\gamma+\mu}\leq 1$, or equivalently $\nu=\beta N-(\gamma+\nu)\leq 0$. If
	\begin{align}\label{extinction condition}
	\limsup_{t\to\infty}\frac{Z_t}{t}\leq 0,\quad\mbox{ almost surely},
	\end{align}
	then 
	\begin{align*}
	\lim_{t\to\infty}\mathcal{I}_t=0\quad\mbox{ almost surely}.
	\end{align*}
\end{theorem}

\begin{proof}
	We take $G(y):=\ln\left(\frac{y}{N-y}\right)$, for $y\in]0,N[$; an application of the It\^o formula gives 
	\begin{align*}
		dG(\mathcal{I}_t)=\left[\nu-(\gamma+\mu)\frac{\mathcal{I}_t}{N-\mathcal{I}_t}\right]dt+NdZ_t.
	\end{align*}
	this yields
	\begin{align*}
	G(\mathcal{I}_t)=\ln\left(\frac{i_0}{N-i_0}\right)+\int\limits_0^tf(\mathcal{I}_s) ds+NZ_t,
	\end{align*}
	where $f$ is defined as (\ref{f(x)}) again. By utilizing the monotonicity of $f$, we get
    \begin{align*}
	\limsup_{t\to\infty}\frac{1}{t}G(\mathcal{I}_t)<\nu+N\limsup_{t\to\infty}\frac{Z_t}{t}.
	\end{align*}
	Now, if 
	\begin{align*}
	\limsup_{t\to\infty}\frac{Z_t}{t}\leq 0,\quad\mbox{ almost surely},
	\end{align*}
	then $\nu\leq0$ implies immediately that 
	\begin{align*}
	\limsup_{t\to\infty}\frac{1}{t}\ln\left(\frac{\mathcal{I}_t}{N-\mathcal{I}_t}\right)<0.
	\end{align*}
	The last inequality entails the statement of our theorem and competes the proof.
\end{proof}

We now state the analogue of Theorem \ref{Persistency Theorem}.

\begin{theorem}\label{Generalized Persistence Theorem}
	Assume that $R^D_0>1$, or equivalently $\nu>0$. If
	\begin{align}\label{assumption Z}
	0\leq\liminf_{t\to\infty}\frac{Z_t}{t}\leq\limsup_{t\to\infty}\frac{Z_t}{t}<+\infty,
	\end{align}
	then the infection is persistent. More precisely, for all $x\in ]0,N[$ we have
	\begin{align}
	\liminf_{t\to \infty}\mathcal{I}_t\leq x\leq \limsup_{t \to\infty}\mathcal{I}_t,\quad\mbox{almost surely}.
	\end{align}
\end{theorem}

\begin{proof}
	We now proceed by contradiction as in the proof of Theorem \ref{Persistency Theorem}. Similarly, we take $G(y):=\ln\left(\frac{y}{N-y}\right)$, for $y\in]0,N[$; an application of the It\^o formula gives 
	\begin{align*}
	dG(\mathcal{I}_t)=\left[\nu-(\gamma+\mu)\frac{\mathcal{I}_t}{N-\mathcal{I}_t}\right]dt+NdZ_t.
	\end{align*}
	We now fix $0<x<N$ and first prove that 
	\begin{align*}
	\limsup_{t\to\infty}\mathcal{I}_t\geq x,\quad\mbox{ almost surely}.
	\end{align*}
	Assume to the contrary that the event $\left\{\limsup_{t\to\infty}\mathcal{I}_t<x\right\}$ has positive probability. This means that there exists $T(\omega)\geq 0$ such that $\mathcal{I}_t<x$, for all $t\geq T(\omega)$; this, together with the monotonicity of \eqref{f(x)}, gives
	\begin{align*}
	\liminf_{t\to\infty}\frac{1}{t}\ln\left(\frac{\mathcal{I}_t}{N-\mathcal{I}_t}\right)\geq&\liminf_{t\to\infty}\frac{1}{t}\int_0^tf(\mathcal{I}_s)ds+N\liminf_{t\to\infty}\frac{Z_t}{t}\\
	>&f(x)+N \liminf_{t\to\infty}\frac{Z_t}{t}.
	\end{align*}
	Now, the core of Theorem \ref{Persistency Theorem} is the left hand side above being positive and this is guaranteed if
	\begin{eqnarray}\label{limsup condition persistency generalization}
	N\liminf_{t\to\infty}\frac{Z_t}{t}\geq-f(x).
	\end{eqnarray}
	In this case $\liminf_{t\to\infty}\frac{1}{t}\ln\left(\frac{\mathcal{I}_t}{N-\mathcal{I}_t}\right)>0$ and hence $\lim_{t\to\infty}\mathcal{I}_t=N$, contradicting our initial assumption. Since the range of $f$ is $]-\infty,\nu]$, condition (\ref{limsup condition persistency generalization}) is implied by the first inequality in (\ref{assumption Z}) and positivity of $\nu$.\\
	We now prove that 
	\begin{align*}
	\liminf_{t\to\infty}\mathcal{I}_t\leq x,\quad\mbox{almost surely}. 
	\end{align*}
	Assume to the contrary that the event $\left\{\liminf_{t\to\infty}\mathcal{I}_t> x\right\}$ has positive probability; this implies the existence of $T(\omega)>0$ such that $\mathcal{I}_t>x$, for all $t\geq T(\omega)$. Then,
	\begin{align*}
	\limsup_{t\to\infty}\frac{1}{t}\ln\left(\frac{\mathcal{I}_t}{N-\mathcal{I}_t}\right)&\leq\limsup_{t\to\infty}\int_0^tf(\mathcal{I}_s)ds+N\limsup_{t\to\infty}\frac{Z_t}{t}\\
	&<f(x)+N\limsup_{t\to\infty}\frac{Z_t}{t}.
	\end{align*}
	Now, the core of Theorem \ref{Persistency Theorem} is the left hand side above being negative and this is guaranteed if 
	\begin{align}
		\label{liminf condition persistence general}
		N\limsup_{t\to\infty}\frac{Z_t}{t}\leq-f(x)
	\end{align}
	In this case $\limsup_{t\to\infty}\frac{1}{t}\ln\left(\frac{\mathcal{I}_t}{N-\mathcal{I}_t}\right)<0$ and hence $\lim_{t\to\infty}\mathcal{I}_t=0$ which contradicts our initial assumption. Since the range of $f$ is $]-\infty,\nu]$, condition (\ref{liminf condition persistence general}) is implied by the last inequality in (\ref{assumption Z}).
\end{proof}

\begin{remark}
	If the process $\{Z_t\}_{t\geq 0}$ satisfies both \eqref{extinction condition} and \eqref{assumption Z}, then for $\nu\leq 0$ one has extinction and for $\nu>0$ one has persistence for the associated model. It is useful to note this is the case only when
	\begin{align}\label{remark general extinction and persistence}
		\lim_{t\to\infty}\frac{Z_t}{t}=0.
	\end{align}
	\label{remark}
\end{remark}

\subsection{Trajectory simulations}

\begin{figure}[b!]
	\centering
	\begin{subfigure}[c]{0.49\textwidth}
		\includegraphics[width=\linewidth]{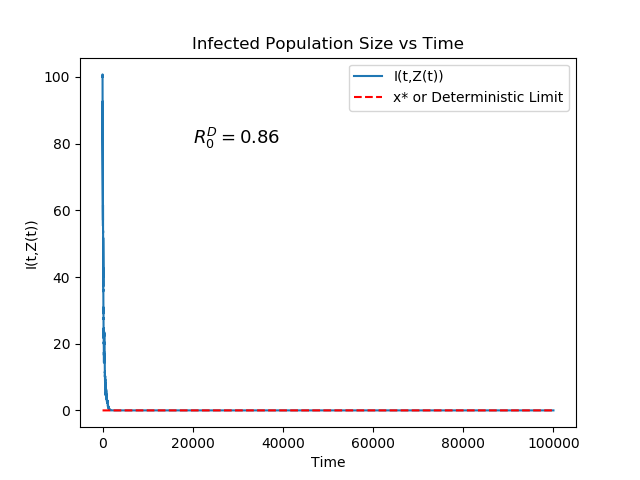}
		\caption{$R^D_0=0.857$}
		\label{Generalization negative alpha R^D_0=0.857}
	\end{subfigure}
	\begin{subfigure}[c]{0.49\textwidth}
		\includegraphics[width=\linewidth]{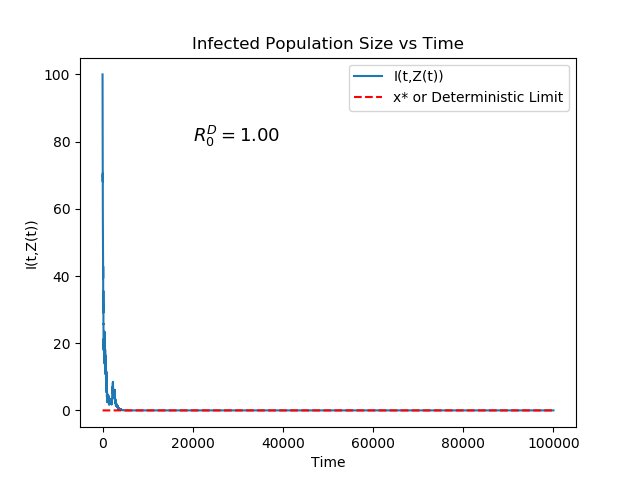}
		\caption{$R^D_0=1.000$}
		\label{Generalization negative alpha R^D_0=1.000}
	\end{subfigure}
	\newline
	\begin{subfigure}[c]{0.49\textwidth}
		\includegraphics[width=\linewidth]{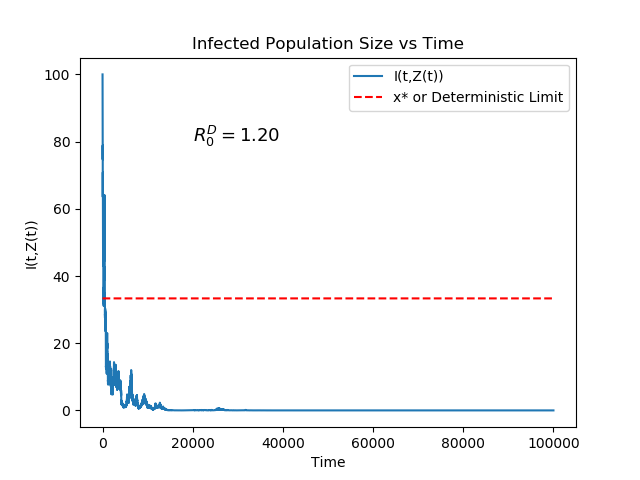}
		\caption{$R^D_0=1.200$}
		\label{Generalization negative alpha R^D_0=1.200}
	\end{subfigure}
	\begin{subfigure}[c]{0.49\textwidth}
		\includegraphics[width=\linewidth]{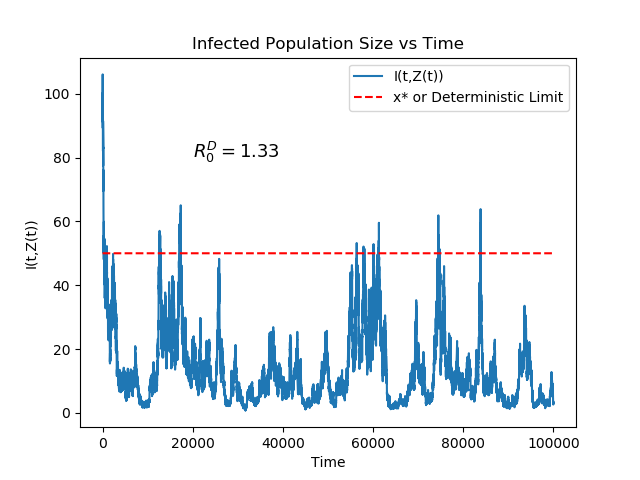}
		\caption{$R^D_0=1.333$}
		\label{Generalization negative alpha R^D_0=1.333}
	\end{subfigure}
	\caption{The plot with parameters $N=200$, $i_0=100$, $\alpha=-0.011$ and $\sigma=0.005$. The label for y-axis $I(t,Z_t)$ stands for $\mathcal{I}_t$.}
	\label{Generalization negative alpha}
\end{figure}

In this section a simple example will be given to support the theoretical results of Theorems \ref{Generalization Extinction Theorem} and \ref{Generalized Persistence Theorem}. We choose
\begin{align}\label{Z(t)}
	Z_t:=\alpha t+\sigma B_t
\end{align}
and note that
\begin{align*}
\limsup\limits_{t\to\infty}\frac{Z_t}{t}=\liminf\limits_{t\to\infty}\frac{Z_t}{t}=\alpha,\quad\mbox{ almost surely}.
\end{align*}
If $\alpha<0$, then the assumption of Theorem \ref{Generalization Extinction Theorem} is fulfilled while those of Theorem \ref{Generalized Persistence Theorem} are not. This means that for $R^D_0\leq 1$, or equivalently $\nu\leq 0$, the extinction is guaranteed; however, for $R^D_0> 1$, or equivalently $\nu>0$, the persistence of infection is not guaranteed.
\par Figure \ref{Generalization negative alpha} supports this claim. As one can see the panels \ref{Generalization negative alpha R^D_0=0.857} and \ref{Generalization negative alpha R^D_0=1.000} shows the extinction of the infection since Theorem \ref{Generalization Extinction Theorem} is satisfied. However, panels \ref{Generalization negative alpha R^D_0=1.200} and \ref{Generalization negative alpha R^D_0=1.333} shows both the examples of extinction and persistence even though the $\nu>0\:(R^D_0>1)$. Thus one looses the properties of the deterministic model.
\par A complementary analysis can be made for $\alpha>0$. In this case the assumption of Theorem \ref{Generalized Persistence Theorem} will be satisfied while those of Theorem \ref{Generalization Extinction Theorem} are not. This indicates that the infection will be persistent as long as $R^D_0>1$; however, the extinction of infection is not guaranteed for $R^D_0\leq1$. The results of Figure \ref{Generalization positive alpha} support this claim.
\begin{figure}[h]
	\centering
	\begin{subfigure}[c]{0.49\textwidth}
		\includegraphics[width=\linewidth]{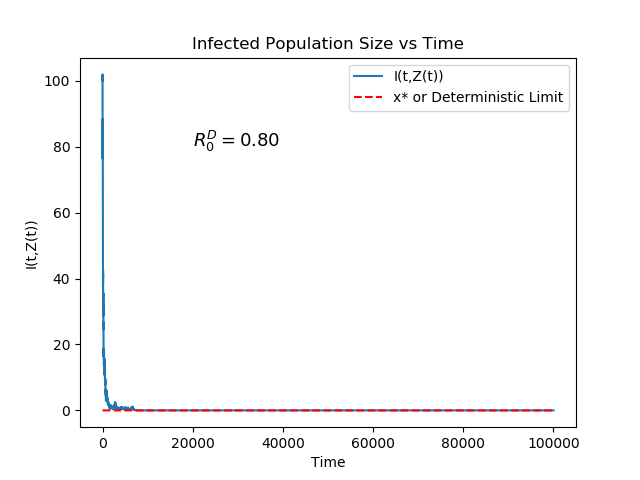}
		\caption{$R^D_0=0.800$}
		\label{Generalization positive alpha R^D_0=0.800}
	\end{subfigure}
	\begin{subfigure}[c]{0.49\textwidth}
		\includegraphics[width=\linewidth]{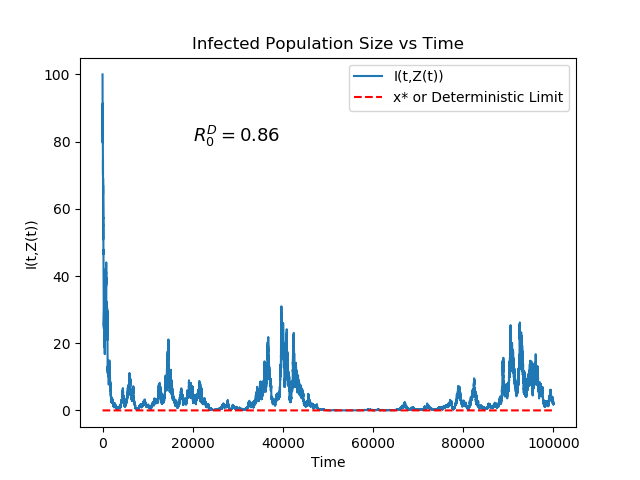}
		\caption{$R^D_0=0.857$}
		\label{Generalization positive alpha R^D_0=0.857}
	\end{subfigure}
	\newline
	\begin{subfigure}[c]{0.49\textwidth}
		\includegraphics[width=\linewidth]{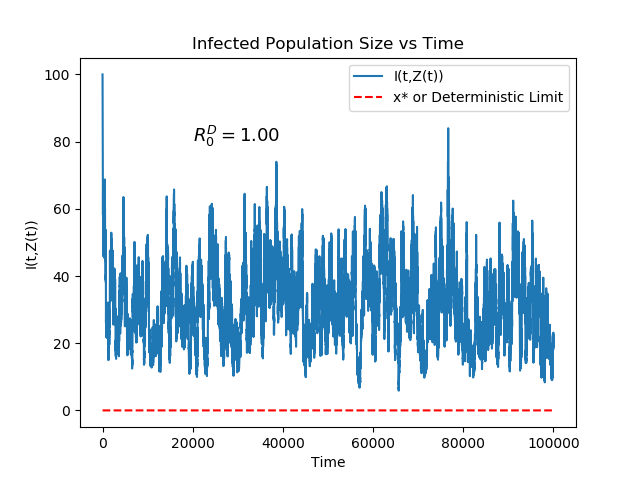}
		\caption{$R^D_0=1.000$}
		\label{Generalization positive alpha R^D_0=1.000}
	\end{subfigure}
	\begin{subfigure}[c]{0.49\textwidth}
		\includegraphics[width=\linewidth]{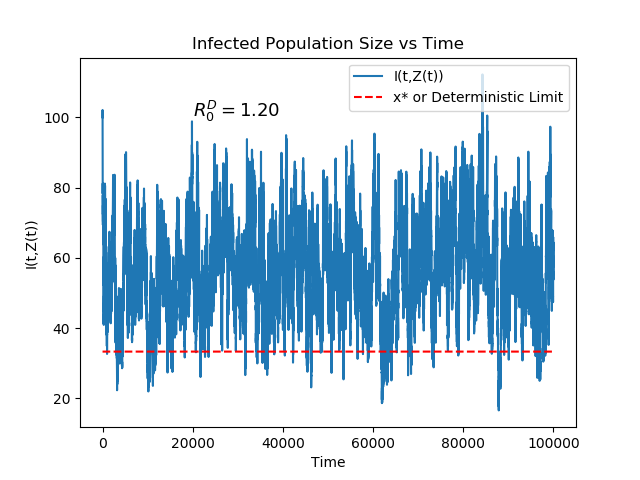}
		\caption{$R^D_0=1.200$}
		\label{Generalization positive alpha R^D_0=1.200}
	\end{subfigure}
	\caption{The plot with parameters $N=200$, $i_0=100$, $\alpha=-0.011$ and $\sigma=0.005$. The label for y-axis $I(t,Z_t)$ stands for $\mathcal{I}_t$.}
	\label{Generalization positive alpha}
\end{figure}
\par In the panel \ref{Generalization positive alpha R^D_0=1.200}, the infection is persistent as expected. However, the extinction is not seen in panels \ref{Generalization positive alpha R^D_0=1.000} and \ref{Generalization positive alpha R^D_0=0.857} since the sufficient condition in Theorem \ref{Generalization Extinction Theorem} is not satisfied. Because the condition was sufficient and not necessary, although it is not satisfied, the extinction can still happen for $R^D_0\leq1$, as shown in panel \ref{Generalization positive alpha R^D_0=0.800}.
\par One last important comment is that, when $\alpha=0$, then one has the same kind of perturbation of the $\beta$ parameter utilized in \cite{Gray et. al.}, so now $Z_t=\sigma B_t$. If the $\beta$ parameter is perturbed in the same way as this study and not in \cite{Gray et. al.}, since $\lim\limits_{t\to\infty}Z_t/t=0$, by the remark \ref{remark} it can be said that, for $R^D_0\leq1\:(\nu\leq0)$ cases there will be extinction and for cases $R^D_0>1\:(\nu>0)$ there will be persistence of infection. We would like to stress that conditions are not the same as published in \cite{Gray et. al.} because the perturbation method is different in this study.

\section{Discussion}

In this study we propose a new perturbation method for the disease transmission coefficient in SIS model. Our approach consists in acting directly on the explicit solution of the deterministic problem, thus avoiding delicate manipulations of differential quantities such as $dB_t$. Nevertheless, our model is cross validated at the level of differential equations once we smooth the perturbation and use the Wong-Zakai theorem. We first use this method with a perturbation of mean reverting type and properties of the corresponding model are analysed. Then, generalization to different sources of randomness are investigated.\\
When the deterministic SIS model \ref{Deterministic SIS Model} is perturbed with a mean reverting Ornstein-Uhlenbeck process $\{Y_t\}_{t\geq 0}$, the solution $\{\mathtt{I}_t\}_{t\geq 0}$ is shown to preserve the deterministic model's regimes for extinction and persistence. Namely:
\begin{itemize}
\item if $R^D_0\leq 1$, then we have extinction of infection;
\item if $R^D_0> 1$, then we have persistence of infection.	
\end{itemize}
Then, we identified some simple sufficient conditions on the class of possible perturbations which entail the same key feature for the corresponding models.
This study emphasizes that the methodology of perturbation of deterministic models is crucial in generating different stochastic versions. For further studies one can apply the same rationale to other parameters in the deterministic SIS model, i.e. $\gamma+\mu$, or to other models, such as Lotka-Volterra type of equations (\cite{Lotka},\cite{Volterra}).


\begin{thebibliography}{9}
	
	\bibitem{Allen}
	Allen, E., Modelling with Itô Stochastic Differential Equations, \textit{Springer-Verlag}, 2007.
	
	\bibitem{BL}
	E. Bernardi and A.Lanconelli, A note about the invariance of the basic reproduction number for stochastically perturbed SIS models, \emph{arXiv:2005.04973v2}.
	
	\bibitem{Brauer}
	F. Brauer, L.J.S. Allen, P. Van den Driessche and J. Wu, Mathematical Epidemiology, \emph{Lecture Notes in Mathematics}, No. 1945, Mathematical Biosciences Subseries,
	2008.
	
	\bibitem{Chinese Mean Reverting Paper}
	Cai, Y., Jiao, J., Gui, Z., Liu, Y., \& Wang, W. (2018). Environmental variability in a stochastic epidemic model. \textit{Applied Mathematics and Computation}, 329, 210-226. doi:10.1016/j.amc.2018.02.009
	
	\bibitem{Gray et. al.}
	A. Gray, D. Greenhalgh, L. Hu, X. Mao, J. Pan, A stochastic differential equation SIS epidemic model, \emph{SIAM J. Appl. Math.} 71 (3) (2011) 876–902.
	
	\bibitem{Gnorrhea}
	Hethcote, H.W. and Yorke, J.A., Gonorrhea Transmission Dynamics and Control, \textit{Lecture Notes in Biomathematics 56}, Springer-Verlag, 1994.
	
	\bibitem{SIRS Paper}
	Jin, Y., Wang, W., \& Xiao, S. (2007). An sirs model with a nonlinear incidence rate. \textit{Chaos, SolItôns \& Fractals}, 34(5), 1482-1497. doi:10.1016/j.chaos.2006.04.022
	
	\bibitem{KS}
	I. Karatzas and S. E. Shreve, \emph{Brownian motion and stochastic calculus}, Springer-Verlag, New York, 1991.
	
	\bibitem{Lotka}
	A. J. Lotka, Contribution to quantitative parasitology, \emph{J. Wash. Acad. Sci.} 13 (1923) 152-158.
	
	\bibitem{Mao}
	Mao, X. (2011). Stochastic differential equations and applications. In \textit{Stochastic differential equations and applications} (Second ed., pp. 12-13). Oxford: Woodhead.
	
	\bibitem{Stroock Varadhan}
	D. W. Stroock and S. R. S. Varadhan, On the support of diffusion processes with applications to the strong maximum principle, \emph{Proceedings $6$-th Berkeley Symposium Math. Statist. Probab.} \textbf{3} (1972) University of California Press, Berkeley, 333-359.
	
	\bibitem{Volterra}
	V. Volterra, Variazioni e fluttuazioni del numero d'individui in specie animali conviventi. \emph{Mem. Acad. Lincei} 2 (1926) 31-113.
	
	\bibitem{Wong-Zakai Paper}
	E. Wong and M. Zakai: On the relation between ordinary and stochastic differential equations, \textit{Intern. J. Engr. Sci.} 3 (1965) 213-229.
	
	\bibitem{Feller's Test}
	Xu, C. (2017). Global threshold dynamics of a stochastic differential equation sis model. \textit{Journal of Mathematical Analysis and Applications}, 447(2), 736-757. doi:10.1016/j.jmaa.2016.10.041
	
    \bibitem{Vaccination Paper}
	Zhou, Y., \& Liu, H. (2003). Stability of periodic solutions for an Sis model with Pulse vaccination. \textit{Mathematical and Computer Modelling}, 38(3-4), 299-308. doi:10.1016/s0895-7177(03)90088-4
	
\end{thebibliography}
\end{document}